\documentclass[12pt, reqno]{amsart}
\usepackage{}
\usepackage{amsmath, amsthm, amscd, amsfonts, amssymb, graphicx, color,cite}
\usepackage[bookmarksnumbered, colorlinks, plainpages]{hyperref}

\newtheorem{thm}{Theorem}[section]

\newtheorem{lem}[thm]{Lemma}
\newtheorem{prop}[thm]{Proposition}
\theoremstyle{definition}
\newtheorem{defn}{Definition}[section]
\theoremstyle{remark}
\newtheorem{rem}[thm]{Remark}
\numberwithin{equation}{section}

\begin{document}

\title[On mixed Quermassintegrals]{On Mixed Quermassintegrals for log-concave Functions}

\author[F. Chen, J. Fang, M. Luo, C. Yang]{Fangwei Chen$^1$, Jianbo Fang$^1$, Miao Luo$^2$, Congli Yang $^2$}

\address{1. School of Mathematics and Statistics, Guizhou University of Finance and Economics,
Guiyang, Guizhou 550025, People's Republic of China.}
\email{cfw-yy@126.com}
\email{16995239@qq.com}

\address{2. School of Mathematical Sciences, Guizhou Normal
University, Guiyang, Guizhou 550025, People's Republic of China}
\email{lm975318@126.com}
\email{yangcongli@gznu.edu.cn}

\thanks{The  work is supported in part by CNSF (Grant No. 11561012, 11861004, 11861024), Guizhou  Foundation for Science and Technology (Grant
No. [2019] 1055, [2019]1228), Science and technology top talent support program of Guizhou Eduction Department (Grant No. [2017]069).}




\dedicatory{}



\begin{abstract}
In this paper, the functional Quermassintegrals of log-concave functions in $\mathbb R^n$ are discussed, we obtain the integral expression of the $i$-th functional mixed Quermassintegrals, which are similar to the integral expression of the $i$-th Quermassintegrals of convex bodies.

\end{abstract} \maketitle

\section{introduction}

Let $\mathcal K^n$ be the set of convex bodies (compact convex subsets with nonempty interiors) in $\mathbb R^n$, the fundamental Brunn-Minkowski inequality for convex bodies states that for $K,\,L \in\mathcal K^n$,  the volume of the bodies and of their Minkowski sum $K+L=\{x+y:x\in K, \,\, and\,\,y\in L\}$ are given by
\begin{align}\label{bur-min-ineq}
  V\big(K+L)^{\frac{1}{n}}\geq V(K)^{\frac{1}{n}}+V(L)^{\frac{1}{n}},
\end{align}
 with equality if and only if $K$ and $L$ are homothetic, namely they agree up to a translation and a dilation. Another important geometric inequality related to the convex bodies $K$ and $L$ is the mixed volume,  the important result concern the mixed volume is the Minkwoski's first inequality
\begin{align}\label{mix-vol-ineq}
V_1(K, L):=\frac{1}{n}\lim_{t\rightarrow 0^+}\frac{V(K+tL)-V(K)}{t}\geq V(K)^{\frac{n-1}{n}}V(L)^{\frac{1}{n}},
\end{align}
for $K, L\in \mathcal K^n$. Specially, when choose $L$ to be a unit ball, up to a factor, $V_1(K, L)$ is exactly the perimeter of $K$, and inequality (\ref{mix-vol-ineq}) turns out to be the isoperimetric inequality in the class of convex bodies.  The mixed volume $V_1(K, L)$ admits a simple integral representation (see\cite{lut-the1993,lut-the1996})
 \begin{align}
   V_1(K,L)=\frac{1}{n}\int_{ S^{n-1}}h_LdS_K,
 \end{align}
where $h_L$ is the support function of $L$ and $S_K$ is the area measure of $K$.

The Quermassintegrals $W_i(K)$ $(i=0,\,\,1,\,\,\cdots n)$ of $K$, which are defined by letting $W_0(K)=V_n(K)$, the volume of $K$; $W_n(K)=\omega_n$, the volume of the unit ball $B^n_2$ in $\mathbb R^n$ and for general $i=1,\,2,\cdots, n-1$,
\begin{align}
W_{n-i}(K)=\frac{\omega_n}{\omega_i}\int_{\mathcal G_{i,n}}vol_i(K
|_{\xi_i})d\mu(\xi_i),
\end{align}
where the $\mathcal G_{i,n}$ is the Grassmannian manifold of $i$-dimensional linear subspaces of $\mathbb R^n$, $d\mu(\xi_i)$ is the normalized Haar measure on $\mathcal G_{i,n}$, $K|_{\xi_i}$ denotes the orthogonal projection of $K$ onto the $i$-dimensional subspaces $\xi_i$, and $vol_i$ is the $i$-dimensional volume on space $\xi_i$.

In the 1930s, Aleksandrov, Fenchel and  Jessen (see \cite{ale-on1937,fen-jen-men1938}) proved that for a convex body $K$ in $\mathbb R^n$, there exists a regular Borel measure $S_{n-1-i}(K)$ ($i=0,\,1,\cdots, n-1$) on $S^{n-1}$, the unit sphere in $\mathbb R^n$, such that for any convex bodies $K$ and $L$, the following representations hold
\begin{align}\label{quer}
\nonumber W_i(K,L)&=\frac{1}{n-i} \lim_{t \rightarrow 0^+}\frac{W_i(K+t L)-W_i(K)}{\epsilon}\\
&=\frac{1}{n}\int_{S^{n-1}}h_L(u)dS_{n-1-i}(K,u).
\end{align}
The quantity $W_i(K,L)$ is called the $i$-th mixed Quermassintegral of $K$ and $L$.

In the 1960s, the Minkowski addition was extended to the $L^p$ $(p\geq 1)$ Minkowski sum $h^p_{K+_pt \cdot L}=h^p_K+t h^p_L.$
The extension of the mixed Quermassintegrals to the $L^p$ mixed Quermassintegrals due to Lutwak \cite{lut-the1993}, and the $L^p$ mixed Quermassintegral inequalities, the $L^p$ Minkowski problem are established. See\cite{lut-yan-zha-opt2006,wer-on2007,wer-ye-ine2010,lut-yan-zha-vol2004,
wer-ye-new2008,lut-yan-zha-on2004,lut-yan-zha-sha2002,lut-yan-zha-lp2000,lut-the1996,Hab-sch-asy2009} for more about the $L^p$ Minkowski theory and $L^p$ Minkowski inequalities. The first variation of the $L^p$ mixed Quermassintegrals are defined by
\begin{align}\label{lp-que-def}
  W_{p,i}(K, L):=\frac{p}{n-i}\lim_{t \rightarrow 0^+}\frac{W_i(K+_pt \cdot L)-W_i(L)}{t},
\end{align}
for $i=0,\,1,\cdots, n-1$. In particular, for $p=1$ in (\ref{lp-que-def}), it is $W_i(K,L)$, and  $W_{p,0}(K,L)$ is denoted by $V_p(K,L)$, which is called the $L_p$ mixed volume of $K$ and $L$. Similarly, the $L^p$ mixed Quermassintegral has the following integral representation (see \cite{lut-the1993}):
\begin{align}\label{p-que-exp}
  W_{p,i}(K,L)=\frac{1}{n}\int_{S^{n-1}}h_L^p(u)dS_{p,i}(K,u).
\end{align}
The measure $S_{p,i}(K,\cdot)$ is absolutely continuous with respect to $S_i(K,\cdot)$, and has Radon-Nikodym derivative
$\frac{dS_{p,i}(K,\cdot)}{dS_{i}(K,\cdot)}=h_K(\cdot)^{1-p}.$
 Specially, when $p=1$ in (\ref{p-que-exp}) yields the representation (\ref{quer}).

In most recently, the interest in the log-concave functions has been considerably increasing, motivated by the analogy properties between the log-concave functions and the volume convex bodies in $\mathcal K^n$. The classical Pr\'{e}kopa-Leindler inequality (see \cite{lei-on1972,bra-lei-on1976,pre-log1971,pre-on1973,pre-new1975}) firstly shows the connections of the volume of convex bodies and  log-concave functions. The Blaschke-Santal\'{o} inequality for even log-concave functions is  established by Ball in  \cite{bal-iso1986,bal-log1988}, the general case is proved by Artstein-Avidan, Klartag and Milman \cite{art-kla-mil-the2004}, other proofs are given by Fradelizi, Meyer \cite{fra-mey-som2007} and Lehec \cite{leh-a2009,leh-par2009}. See \cite{fra-gor-mey-rei-the2010,bar-bor-fra-sta2014,rot-a2014,had-jim-mon-asy2019} for more about the functional Blaschke-Santal\'o inequalities. The mean width for log-concave function is introduced by Klartag, Milman and Rotem \cite{kla-mil-geo2005,rot-on2012,rot-sup2013}. The affine isoperimetric inequality for log-concave functions are proved by Artstein-Avidan, Klartag, Sch\"{u}tt and Werner \cite{avi-kla-sch-wer-fun2012}. The John ellipsoid for log-concave functions have been establish by Guti\'{e}rrez, Merino Jim\'{e}nez and Villa \cite{gut-mer-jim-vil-joh2018}, the LYZ ellipsoid for log-concave functions are established by Fang and Zhou \cite {fan-zho-lyz2018}. See \cite{gut-mer-jim-vil-rog2016,avi-slo-a2015,Bar-cap-fus-pis-sta2014,cag-wer-div2014,cag-wer-mix2015,cag-ye-aff2016,lin-aff2017} for more about the pertinent results. To establish the functional versions of inequalities and problems from the points of convex geometric analysis is a new research fields.


Let $f=e^{-u}$, $g=e^{-v}$ be log-concave functions, $\alpha,\,\,\beta>0$, the ``sum"  and ``scalar multiplication" of log-concave functions are defined as,
\begin{align*}
  \alpha \cdot f \oplus \beta\cdot g:=e^{-w}, \,\,\,\,\,\,\,\,\mbox {where}\,\,\, w^*=\alpha u^*+\beta v^*,
\end{align*}
here $w^*$ denotes as usual the Fenchel conjugate of the convex function $\omega$.
The total mass integral $J(f)$ of $f$ is defined by,
$J(f)=\int_{\mathbb R^n}f(x)dx.$
 In paper of Colesanti and Fragal\`{a} \cite{col-fra-the2013}, the quantity $\delta J (f,g)$, which is called as the first variation of $J$ at $f$ along $g$,
$\delta J(f,g)=\lim\limits_{t\rightarrow 0^+}\frac{J(f\oplus t\cdot g)-J(f)}{t},$
is discussed. It has been shown that $ \delta J(f,g)$ is finite and is given by
$$\delta J(f,g)=\int_{\mathbb R^n}v^*d\mu(f),$$
where $\mu(f)$ is the measure of $f$ on $\mathbb R^n$.

Inspired by the paper of Colesanti and Fragal\`{a} \cite{col-fra-the2013}, in this paper, we define the $i$-th functional Quermassintegrals $W_i(f)$ as the  $i$-dimensional average total mass of $f$,
$$W_{i}(f):=\frac{\omega_{n}}{\omega_{n-i}}\int_{\mathcal G_{n-i,n}}J_{n-i}(f)d\mu(\xi_{n-i}).$$
Where $J_i(f)$ denotes the $i$-dimensional total mass of $f$ defined in section four, $\mathcal G_{i,n}$ is the Grassmannian manifold of $\mathbb R^n$ and $d\mu(\xi_{n-i})$ is the normalized measure on $\mathcal G_{i,n}$.
Moreover, we define the first variation of $W_i$ at $f$ along $g$, which is
\begin{align*}
   W_i(f,g)=\lim_{t\rightarrow 0^+}\frac{W_i(f\oplus t\cdot g)-W_i(f)}{t}.
\end{align*}
It is natural extension of the Quermassintegrals of convex bodies in $\mathbb R^n$, we call it the $i$-th functional mixed Quermassintegral. In fact, if one takes $f=\chi_K$, and $dom(f)=K\in \mathbb R^n$, then $W_i(f)$ turn out to be $W_i(K)$, and $W_i(\chi_K, \chi_L)$ equals to the $W_i(K,L)$.  The main result in this paper is to show that the $i$-th functional mixed Quermassintegrals have the following integral expressions.
\begin{thm}\label{thm-mix-quer}
Let $f$, $g$ are integrable functions on $\mathcal A'$, $\mu_i(f)$ be the $i$-dimensional measure of $f$, and $W_i(f,g)$ $(i=0, 1, \cdots, n-1)$ are the $i$-th functional mixed Quermassintegrals of $f$ and $g$. Then
\begin{align}\label{mix-que-exp}
W_i(f,g)=\frac{1}{n-i}\int_{\mathbb R^n}h_{g|_{\xi_{n-i}}} d\mu_{n-i}(f),
\end{align}
where $h_g$ is the support function of $g$.
\end{thm}

The paper is organized as follows, in  section 2, we introduce some notations about the log-concave functions. In section 3, the projection of log-concave function are discussed. In section 4, we focus on how can we represent the $i$-th functional mixed Quermassintegrals $W_i(f,g)$ similar as $W_i(K,L)$. Owing to the Blaschke-Petkantschin formula and the similar definition of the support function of $f$, we obtain the integral represent of the $i$-th functional mixed Quermassintegrals $W_i(f,g)$.


\section{preliminaries}

 Let $u:\Omega\rightarrow (-\infty,+\infty]$ be a convex function, that is $u\big((1-t)x+ty\big)\leq (1-t)u(x)+tu(y)$ for $t\in [0,1]$, here $\Omega=\{x\in \mathbb R^n:u(x)\in \mathbb R\}$ is the domain of $u$.
By the convexity of $u$, $\Omega$ is a convex set in $\mathbb R^n$. We say that $u$ is proper if $\Omega\neq \emptyset$, and $u$ is of class $\mathcal C^2_+$ if it is twice differentiable on $int (\Omega) $, with a positive definite Hessian matrix. In the following we define the subclass of $u$,
\begin{align*}
\mathcal L=\big\{u:\Omega\rightarrow &(-\infty,+\infty]: \mbox{u is convex, low semicontinuous}\\
&\mbox{ and}\,\,\lim_{\|x\|\rightarrow +\infty}u(x)=+\infty\big\}.
\end{align*}
Recall that the Fenchel conjugate of $u$ is the convex function defined by
\begin{align}\label{leg-con}
  u^*(y)=\sup_{x\in\mathbb R^n}\big\{\langle x,y\rangle -u(x)\big\}.
\end{align}
It is obvious that $u(x)+u^*(y)\geq \langle x,y\rangle$ for all $x, y\in \mathbb R^n$, and there is an equality if and only if $x\in \Omega$ and $y$ is in the subdifferential of $u$ at $x$, that means
\begin{align}\label{leg-equ}
  u^*(\nabla u(x))+u(x)=\langle x,\nabla u(x)\rangle.
\end{align}
Moreover, if $u$ is a lower semi-continuous convex function, then also $u^*$ is a lower semi-continuous convex function, and $u^{**}=u$.

The infimal convolution of functions $u$ and $v$ from $\Omega$ to $(-\infty, +\infty]$ defined by
\begin{align}\label{inf-con}
  u\Box v(x)=\inf_{y\in \Omega}\big\{u(x-y)+v(y)\big\}.
\end{align}
The right scalar multiplication by a nonnegative real number $\alpha$:
\begin{align}\label{rig-mul}
\big(u\alpha\big)(x):=\left\{
           \begin{array}{ll}
           \alpha u\big(\frac{x}{\alpha}\big), & \hbox{$if\,\, \alpha >0$;} \\
           I_{\{0\}}, & \hbox{$if\,\,\alpha=0$.}
           \end{array}
         \right.
\end{align}

The following proposition below gathers some elementary properties of the Fenchel conjugate and the infimal convolution of $u$ and $v$,  which can be found in \cite{col-fra-the2013,roc-con1970}.
\begin{prop}
 Let $u,\,\,v:\Omega\rightarrow (-\infty,+\infty]$ be convex functions. Then:
\begin{enumerate}
               \item $\big(u\Box v\big)^*=u^*+v^*$;
               \item $(u\alpha)^*=\alpha u^*, \,\,\,\,\,\alpha>0$;
               \item $dom(u\Box v)=dom(u)+dom(v)$;
               \item it holds $u^*(0)=-\inf(u)$, in particular if $u$ is proper, then $u^*(y)>-\infty$; $\inf(u)>-\infty$ implies $u^*$ is proper.
             \end{enumerate}
\end{prop}



The following Proposition about the Fenchel and Legendre conjugates are obtained in \cite{roc-con1970}.
\begin{prop}[\cite{roc-con1970}]\label{leg-con-par}
  Let $u: \Omega\rightarrow (-\infty,+\infty]$ be a closed convex function, and set $\mathcal C:=int (\Omega)$, $\mathcal C^*:=int (dom(u^*))$. Then $(\mathcal C,u)$ is a convex function of Legendre type if and only if $\mathcal C^*, u^*$ is. In this case $(\mathcal C^*,u^*)$ is the Legendre conjugate of $(\mathcal C,u)$ (and conversely). Moreover, $\nabla u:=\mathcal C\rightarrow \mathcal C^*$ is a continuous bijection, and the inverse map of $\nabla u$ is precisely $\nabla u^*$.
\end{prop}

A function $f: \mathbb R^n\rightarrow (-\infty, +\infty]$ is called log-concave if for all $x, y\in \mathbb R^n$ and $0<t<1$, we have $f\big((1-t)x+ty \big)\geq f^{1-t}(x)f^t(y).$
If $f$ is a strictly positive log-concave function on $\mathbb R^n$, then there exist a convex function $u:\Omega\rightarrow (-\infty,+\infty]$ such that $f=e^{-u}$.
The log-concave function is closely related to the convex geometry of $\mathbb R^n$. An example of a log-concave function is the characteristic function $\chi_K$ of a convex body $K$ in $\mathbb R^n$, which is defined by
\begin{align}\label{chi-fun}
  \chi_K(x)=e^{-I_K(x)}=\left\{
                                 \begin{array}{ll}
                                   1, & \hbox{if \,\, $x\in K$;} \\
                                   0, & \hbox{if \,\, $x\notin K$,}
                                 \end{array}
                               \right.
\end{align}
where $I_K$ is a lower semi-continuous convex function, and the indicator function of $K$ is,
\begin{align}\label{ind-fun}
I_K(x)=\left\{
               \begin{array}{ll}
                 0, & \hbox{if \,\, $x\in K$;} \\
                 \infty, & \hbox{if \,\, $x\notin K$.}
               \end{array}
             \right.
\end{align}

In the later sections, we also use $f$ to denote  $f$ been extended to $\mathbb R^n$.
\begin{align}
  \overline{f}=\left\{
                 \begin{array}{ll}
                   f, & \hbox{$x\in \Omega$;} \\
                  0, & \hbox{$x\in\mathbb R^n/\Omega$.}
                 \end{array}
               \right.
\end{align}

Let $\mathcal A=\big\{f:\mathbb R^n\rightarrow (0,+\infty]:\,\,f=e^{-u}, u\in \mathcal L\big\}$ be the subclass of $f$ in $\mathbb R^n$. The addition and multiplication by nonnegative scalars in $\mathcal A$ is defined by (see \cite{col-fra-the2013}).
\begin{defn}
Let $f=e^{-u}$, $g=e^{-v}\in \mathcal A$, and $\alpha, \beta\geq 0$. The sum and multiplication of $f$ and $g$ is defined as
\begin{align*}
  \alpha\cdot f\oplus \beta\cdot g=e^{-[(u\alpha)\Box(v\beta)]}.
\end{align*}
That means
\begin{align}\label{f-plu-g}
\big(\alpha\cdot f\oplus \beta\cdot g\big)(x)=\sup_{y\in\mathbb R^n}f\Big(\frac{x-y}{\alpha}\Big)^\alpha g\Big(\frac{y}{\beta}\Big)^\beta.
\end{align}
\end{defn}
In particularly, when $\alpha=0$ and $\beta>0$, we have $(\alpha\cdot f\oplus \beta\cdot g)(x)=g(\frac{x}{\beta})^{\beta}$; when $\alpha >0$ and $\beta=0$, then $(\alpha\cdot f\oplus \beta\cdot g)(x)=f(\frac{x}{\alpha})^{\alpha}$;  finally, when $\alpha=\beta=0$, we have $\big(\alpha\cdot f\oplus \beta\cdot g\big)=I_{\{0\}}$.

The following Lemma is obtained in \cite{col-fra-the2013}.
\begin{lem}[\cite{col-fra-the2013}]\label{con-for-u}
  Let $u\in \mathcal L$, then there exist constants $a$ and $b$, with $a>0$, such that, for $x\in \Omega$
\begin{align}
  u(x)\geq a\|x\|+b.
\end{align}
Moreover $u^*$ is proper, and satisfies $u^*(y)>-\infty$, $\forall y\in\Omega$.
\end{lem}
The Lemma \ref{con-for-u} grants that $\mathcal L$ is  closed under the operations of infimal convolution and right scalar multiplication defined in (\ref{inf-con}) and (\ref{rig-mul}) are closed.

\begin{prop}[\cite{col-fra-the2013}]\label{clo-uv}
Let $u$ and $v$ belong both to the same class $\mathcal L$, and $\alpha,\,\,\beta\geq 0$. Then $u\alpha \Box v\beta$  belongs to the same class as $u$ and $v$.
\end{prop}

Let $f\in \mathcal A$, according to papers of \cite{avi-mil-a2010,rot-on2012},  the support function of $f=e^{-u}$ is defined as,
\begin{align}
  h_f(x)=(-\log f(x))^*=u^*(x),
\end{align}
here the $u^*$ is the Legendre transform of $u$. The definition of $h_f$ is a proper generalization of the support function $h_K$. In fact, one can easily checks $h_{\chi_K}=h_K$.
Obviously, the support function $h_f$ share the most of the important properties of support functions $h_K$. Specifically, it is easy to check that the function $h:\mathcal A\rightarrow \mathcal L$ has the following properties \cite{rot-sup2013}:
\begin{enumerate}
  \item  $h$ is a bijective map from $\mathcal A\rightarrow \mathcal L.$
  \item $h$ is order preserving: $f\leq g$ if and only if $h_{f}\leq h_{g}.$
  \item $h$ is additive: for every $f, g\in \mathcal A$ we have $h_{f\oplus g}=h_f+h_g.$
\end{enumerate}

The following proposition shows that $h_f$ is $GL(n)$ covariant.
\begin{prop}[\cite{fan-zho-lyz2018}]
  Let $f\in \mathcal A$. For $A\in GL(n)$ and $x\in \mathbb R^n$, then
  \begin{align}
    h_{f\circ A}(x)=h_f(A^{-t}x).
\end{align}
\end{prop}

Let $u,\,v\in \mathcal L$, denote by $u_t=u\Box vt$ $(t>0)$, and $f_t=e^{-u_t}$. The following Lemmas describe the monotonous and convergence of $u_t$ and $f_t$, respectively.
\begin{lem}[\cite{col-fra-the2013}]\label{mon-ut-ft}
  Let $f=e^{-u}$, $g=g^{-v}\in \mathcal A$. For $t>0$, set $u_t=u\square (vt)$ and $f_t=e^{-u_t}$. Assume that $v(0)=0$, then for every fixed $x\in\mathbb R^n$, $u_t(x)$ and $f_t(x)$ are respectively pointwise decreasing and increasing with respect to $t$; in particular it holds
\begin{align}
  u_1(x)\leq u_t(x)\leq u(x) \,\,\,\mbox{and}\,\, f(x)\leq f_t(x)\leq f_1(x)\,\,\,\,\forall x\in \mathbb R^n \,\,\, \forall t\in [0, 1].
\end{align}
\end{lem}

\begin{lem}[\cite{col-fra-the2013}]\label{lem-lim-uv}
  Let $u$ and $v$ belong both to the same class $\mathcal L$ and, for any $t>0$, set $u_t:=u\Box (vt)$. Assume that $v(0)=0$, then
\begin{enumerate}
  \item $\forall x\in \Omega$, $\lim\limits_{t\rightarrow 0^+}u_t(x)=u(x)$;
  \item $\forall E\subset \subset \Omega$, $\lim\limits_{t\rightarrow 0^+}\nabla u_t(x)=\nabla u$ uniformly on $E$.
\end{enumerate}
\end{lem}

\begin{lem}[\cite{col-fra-the2013}]
  Let $u$ and $v$ belong both to the same class $\mathcal L$ and for any $t>0$, let $u_t:=u\Box (vt)$. Then $\forall x\in int(\Omega_t)$, and $\forall t>0$,
\begin{align}
\frac{d}{dt}\big(u_t(x)\big)=-\psi(\nabla u_t(x)),
\end{align}
where $\psi:=v^*$.
\end{lem}

\section{Projection of functions onto linear subspace}


Let $\mathcal G_{i,n}$ $(0\leq i\leq n)$ be the Grassmannian manifold of $i$-dimensional linear subspace of $\mathbb R^n$. The elements of $\mathcal G_{i,n}$ will usually be denoted by $\xi_i$ and, $\xi^\perp_i$ stands for the orthogonal complement of $\xi_i$ which is a $(n-i)$-dimensional subspace of $\mathbb R^n$.
Let $\xi_i\in \mathcal G_{i,n}$ and $f:\mathbb R^n\rightarrow \mathbb R$. The projection of $f$ onto $\xi_i$ is defined by (see  \cite{kla-mil-geo2005,gut-avi-mer-jim-vil-rog2019})
\begin{align}\label{pro-f}
  f|_{\xi_i}(x):=\max\{f(y): y\in x+{\xi_i}^\perp\},\, \,\,\,\,\,\, \forall x\in \Omega|_{\xi_i}.
\end{align}
where $\xi_i^\perp$ is the orthogonal complement of $\xi_i$ in $\mathbb R^n$, $\Omega|_{\xi_i}$ is the projection of $\Omega$ onto $\xi_i$.
By the definition of the log-concave function $f=e^{-u}$, for every $x\in \Omega|_{\xi_i}$,
one can rewrite (\ref{pro-f}) as
\begin{align}\label{pro-u}
  f|_{\xi_i}(x)=\exp\Big\{\max\{-u(y):y\in x+\xi_i^\perp\}\Big\}=e^{-u|_{\xi_i}}(x).
\end{align}
Regards the the ``sum'' and ``multiplication'' of $f$, we say that the projection keeps the structure on $\mathbb R^n$. In other words, we have the following Proposition.
\begin{prop}\label{f-add-mul}
 Let $f,\,\,g\in \mathcal A$, $\xi_i\in\mathcal G_{i,n}$ and $\alpha, \beta>0$. Then
\begin{align}
(\alpha\cdot f\oplus \beta \cdot g)|_{\xi_i}=\alpha\cdot f|_{\xi_i}\oplus\beta\cdot g|_{\xi_i}.
\end{align}
\end{prop}
\begin{proof}
Let $f,\,\,g\in \mathcal A$, let $x_1,\,x_2,\, x\in \xi_i$ such that $x=\alpha x_1+\beta x_2$, then  we have,
\begin{align*}
(\alpha\cdot f\oplus \beta \cdot g)|_{\xi_i}(x)&\geq(\alpha\cdot f\oplus \beta \cdot g)(\alpha x_1+\beta x_2+\xi_i^\perp)\\
&\geq f(x_1+\xi_i^\perp)^\alpha g(x_2+\xi_i^\perp)^\beta.
\end{align*}
Taking the supremum of the second right hand inequality over all $\xi_i^\perp$ we obtain $
(\alpha\cdot f\oplus \beta \cdot g)|_{\xi_i}\geq\alpha \cdot f|_{\xi_i}\oplus \beta\cdot g|_{\xi_i}.$
On the other hand, for $x\in \xi_i$, any $x_1, \,x_2\in \xi_i$ such that $x_1+x_2=x$, then
\begin{align*}
\big(\alpha\cdot f|_{\xi_i}\oplus\beta\cdot g|_{\xi_i}\big)(x)&= \sup_{x_1+x_2=x}\Big\{\max \big\{f^\alpha(\frac{x_1}{\alpha}+\xi_i^\perp)\big\}\max\big\{g^\beta(\frac{x_2}{\beta}+\xi_i^\perp)\big\}\Big\}\\
&\geq  \sup_{x_1+x_2=x}\Big\{\max \Big(f^\alpha(\frac{x_1}{\alpha}+\xi_i^\perp)g^\beta(\frac{x_2}{\beta}+\xi_i^\perp)\Big)\Big\}\\
&=\max\Big\{\sup_{x_1+x_2=x} \Big(f^\alpha(\frac{x_1}{\alpha}+\xi_i^\perp)g^\beta(\frac{x_2}{\beta}+\xi_i^\perp)\Big)\Big\}\\
&=(\alpha\cdot f\oplus \beta \cdot g)|_{\xi_i}(x).
\end{align*}
Here since $f, \,\,g\geq 0$, the inequality $\max\{f\cdot g\}\leq \max\{f\}\cdot\max\{g\}$ holds for $x\in \mathbb R^n$. So we complete  the proof of the result.
\end{proof}
\begin{prop}\label{kep-ord}
  Let $\xi_i\in \mathcal G_{i,n}$, $f$ and $g$ are functions on $\mathbb R^n$, such that $f(x)\leq g(x)$ holds. Then
\begin{align}
f|_{\xi_i}\leq g|_{\xi_i},
\end{align}
holds for any $x\in \xi_i$.
\end{prop}
\begin{proof}
For $y\in x+\xi_i^\perp$, since $f(y)\leq g(y)$, then $f(y)\leq \max\{g(y): y\in x+\xi_i^\perp\}$. So, $\max\{f(y):y\in x+L^\perp_i\}\leq \max\{g(y):y\in x+\xi_i^\perp\}$.
By the definition of the projection, we complete the proof.
\end{proof}
For the convergence of $f$ we have the following.
\begin{prop}\label{lem-lim-uvp}
Let $\{f_i\}$ are functions such that $\lim\limits_{n\rightarrow \infty}f_n=f_0$, $\xi_i\in\mathcal G_{i,n}$, then $\lim\limits_{n\rightarrow \infty}(f_n|_{\xi_i})=f_0|_{\xi_i}$.
\end{prop}
\begin{proof}
 Since $\lim\limits_{n\rightarrow \infty}f_n=f_0$, it means that, for $\forall \epsilon >0$, there exist $N_0$, $\forall n>N_0$, such that $f_0-\epsilon \leq f_n\leq f_0+\epsilon$. By the monotonicity of the projection, we have $f_0|_{\xi_i}-\epsilon \leq f_n|_{\xi_i}\leq f_0|_{\xi_i}+\epsilon$. Hence each $\{f_n|_{\xi_i}\}$ has a convergent subsequence, we  denote it also by $\{f_n|_{\xi_i}\}$, converging to some $f'_0|_{\xi_i}$. Then for $x\in \xi_i$, we have
\begin{align*}
  f_0|_{\xi_i}(x)-\epsilon \leq f'_0|_{\xi_i}(x)=\lim_{n\rightarrow \infty}(f_n|_{\xi_i})(x)&\leq f_0|_{\xi_i}(x)+\epsilon.
\end{align*}
By the arbitrarily of $\epsilon$ we have $f'_0|_{\xi_i}=f_0|_{\xi_i}$, so we complete the proof.
\end{proof}
Combining with Proposition \ref{lem-lim-uvp} and  Proposition \ref{lem-lim-uv}, it is easy to obtain the following Propsosition.
\begin{prop}
 Let $u$ and $v$ belong both to the same class $\mathcal L$, $\Omega\in \mathbb R^n$ be the domain of $u$, for any $t>0$, set $u_t=u\Box (v t)$. Assume that $v(0)=0$ and
 $\xi_i\in\mathcal G_{i,n}$, then
\begin{enumerate}
  \item $\forall x\in  \Omega|_{\xi_i}$, $\lim\limits_{t\rightarrow 0^+}u_t|_{\xi_i}(x)=u|_{\xi_i}(x)$,
  \item $\forall x\in int(\Omega|_{\xi_i}),\,\, \lim\limits_{t\rightarrow 0^+}\nabla u_t|_{\xi_i}=\nabla u|_{\xi_i}$.
\end{enumerate}
\end{prop}

Now let us introduce some fact about the functions $u_t=u\Box (vt)$ with respect to the parameter $t$.
\begin{lem} \label{lem-duvp}
Let $\xi_i\in \mathcal G_{i,n}$, $u$ and $v$ belong both to the same class $\mathcal L$, $u_t:=u\Box (vt)$ and $\Omega_t$ be the domain of $u_t$ ($t>0$). Then for $ x\in{\Omega_t}|_{\xi_i}$,
\begin{align}
\frac{d}{dt}\Big (u_t|_{\xi_i}\Big)(x)=-\psi\Big(\nabla\big( u_t|_{\xi_i}\big)(x)\Big),
\end{align}
where $\psi:=v^*|_{\xi_i}$.
\end{lem}
\begin{proof}

Set $D_t:={\Omega_t}|_{\xi_i}\subset \xi_i$, for fixed $x\in int(D_t)$, the map $t\rightarrow \nabla \big(u_t|_{\xi_i}\big)(x)$ is differentiable on $(0,+\infty)$. Indeed, by the definition of Fenchel conjugate and the definition of projection $u$, it is easy to see that $(u|_{\xi_i})^*=u^*|_{\xi_i}$ and $(u\Box ut)|_{\xi_i}=u|_{\xi_i}\Box ut|_{\xi_i}$ hold. The Lemma \ref{clo-uv} and the property of the projection grant the differentiability. Set  $\varphi:=u^*|_{\xi_i}$ and $\psi:=v^*|_{\xi_i}$, and $\varphi_t=\varphi+t\psi$, then $\varphi_t$  belongs to the class $\mathcal C^2_+$ on $\xi_i$. Then
$\nabla^2 \varphi_t=\nabla^2\varphi+t\nabla^2 \psi$ is nonsingular on $\xi_i$. So the equation
\begin{align}\label{nab-fun-uv}
\nabla \varphi(y)+t\nabla \psi(y)-x=0,
\end{align}
 locally defines a map $y=y(x,t)$ which is of class $\mathcal C^1$. By Proposition \ref{leg-con-par}, we have $\nabla (u_t|_{\xi_i})$ is the inverse map of $\nabla \varphi_t$, that is $\nabla \varphi_t(\nabla (u_t|_{\xi_i}(x))=x$, which means that for every $x\in int(D_t)$  and every $t>0$, $t\rightarrow \nabla (u_t|_{\xi_i})$ is differentiable.
Using the equation (\ref{leg-equ}) again, we have
\begin{align}
  u_t|_{\xi_i}(x)=\big\langle x, \nabla(u_t|_{\xi_i})(x)\big\rangle-\varphi_t\big(\nabla(u_t|_{\xi_i})(x)\big),\,\,\,\,\,\,\,\,\,\,\,\forall x\in int(D_t).
\end{align}
Moreover, note that $\varphi_t=\varphi+t\psi$ we have
\begin{align*}
  u_t|_{\xi_i}(x)&=\Big\langle x, \nabla(u_t|_{\xi_i})(x)\Big\rangle-\varphi\Big(\nabla \big(u_t|_{\xi_i}\big)(x)\Big)-t\psi\Big(\nabla\big( u_t|_{\xi_i}\big)(x)\Big)\\
&=u_t|_{\xi_i}\Big(\nabla\big( u_t|_{\xi_i}\big)(x)\Big)-t\psi\Big(\nabla\big( u_t|_{\xi_i}\big)(x)\Big).
\end{align*}
Differential the above formal we obtain, $\frac{d}{dt}\Big(u_t|_{\xi_i}\Big)(x)=-\psi\Big(\nabla\big( u_t|_{\xi_i}\big)(x)\Big).$
Then we complete the proof of the result.\end{proof}


\section{Functional Quermassintegrals of Log-concave Function}

A function $f\in \mathcal A$ is non-degenerate and integrable if and only if
$\lim\limits_{\|x\|\rightarrow +\infty}\frac{u(x)}{\|x\|}=+\infty.$
Let $\mathcal L'=\big\{u\in\mathcal L:  u\in \mathcal C^2_+(\mathbb R^n),\lim\limits_{\| x\|\rightarrow +\infty}\frac{u(x)}{\|x\|}=+\infty\big\},$
and $\mathcal A'=\big\{f:\mathbb R^n\rightarrow (0,+\infty]:\,\,f=e^{-u}, u\in \mathcal L'\big\}.$ Now we define the $i$-th total mass of $f$.

\begin{defn}
Let $f\in \mathcal A'$, $\xi_{i}\in \mathcal G_{i,n}$ $(i=1,\,2,\cdots,n-1)$, and $x\in \Omega|_{\xi_{i}}$. The $i$-th total mass of $f$ is defined as
\begin{align}\label{i-tot-mass}
  J_{i}(f):=\int_{\xi_{i}}f|_{\xi_{i}}({x})dx,
\end{align}
where $f|_{\xi_i}$ is the projection of $f$ onto $\xi_i$ defined by (\ref{pro-f}), $dx$ is the $i$-dimensional volume element in $\xi_i$.
\end{defn}
\begin{rem}
(1) The definition of the $J_i(f)$ follows the $i$-dimensional volume of the projection a convex body.
If $i=0$, we defined  $J_0(f):=\omega_n$, the volume of the unit ball in $\mathbb R^n$, for the completeness.

(2) When take $f=\chi_K$, the characteristic function of a convex body $K$, one has $J_i(f)=V_i(K)$, the $i$-dimensional volume in $\xi_i$.
\end{rem}
\begin{defn}
Let $f\in\mathcal A'$. Set $\xi_i\in\mathcal G_{i,n}$ be a linear subspace and, for $x\in \Omega|_{\xi_i}$, the $i$-th functional Quermassintegrals of $f$ (or the $i$-dimensional mean projection mass of $f$) are defined as
\begin{align}\label{quer-def}
  W_{n-i}(f):=\frac{\omega_n}{\omega_i}\int_{\mathcal G_{i,n}}J_{i}(f)d\mu(\xi_{i}),
\end{align}
where $J_i(f)$ is the $i$-th total mass of $f$ defined by (\ref{i-tot-mass}), $d\mu(\xi_i)$ is the normalized Haar measure on $\mathcal G_{i,n}$.
\end{defn}
\begin{rem}

(1) The definition of the $W_i(f)$ follows the definition of the $i$-th Quermassintegral $W_i(K)$, that is, the $i$-th mean total mass of $f$ on $\mathcal G_{i,n}$.  Also in the recently paper  \cite{bob-col-fra-que2014}, the authors give the same definition by defining the Quermassintegral of the support set for the quasi-concave functions.

(2) When $i$ equals to $n$ in (\ref{quer-def}), we have $W_0(f)=\int_{\mathbb R^n}f(x)dx=J(f),$
 the total mass function of $f$ defined by Colesanti and Fragal\'a \cite{col-fra-the2013}. Then we can say that our definition of the $W_i(f)$ is a nature extension of the total mass function of $J(f)$.

(3) Form the definition of the Quermassintegrals $W_i(f)$, the following properties are obtained (see also \cite{bob-col-fra-que2014}).
\begin{itemize}
  \item Positivity. $0\leq W_i(f)\leq +\infty.$
  \item Monotonicity. $W_i(f)\leq W_i(g)$, if $f\leq g$.
  \item Generally speaking, the $W_i(f)$ has no homogeneity under dilations. That is $W_i(\lambda \cdot f)=\lambda^{n-i}W_i(f^\lambda),$ where $\lambda \cdot f(x)=\lambda f(x/\lambda), \lambda>0$.
\end{itemize}
 \end{rem}




\begin{defn}
Let $f$, $g\in \mathcal A'$, $\oplus$ and $\cdot$ denote the operations of ``sum'' and ``multiplication'' in $\mathcal A'$. $W_i(f)$ and $W_i(g)$  are, respectly, the $i$-th Quermassintegrals of $f$ and $g$. Whenever the following limit exists
\begin{align}
W_i(f,g)=\frac{1}{(n-i)}\lim_{t\rightarrow 0^+}\frac{W_i(f\oplus t\cdot g)-W_i(f)}{t},
\end{align}
we denote it by $W_i(f,g)$, and call it as the first variation of $W_i$ at $f$ along $g$, or the $i$-th functional mixed Quermassintegrals of $f$ and $g$.
\end{defn}

\begin{rem}
Let $f=\chi_K$ and $g=\chi_L$, with $K$, $L\in \mathcal K^n$. In this case $W_i(f\oplus t\cdot g)=W_i(K+t L)$, then $W_i(f,g)=W_i(K, L)$.
In general, $W_i(f,g)$ has no analog properties of $W_i(K,L)$, for example, $W_i(f,g)$ is not always nonnegative and finite.
\end{rem}

The following is devote to prove that $W_i(f,g)$ exist under the fairly weak hypothesis. First, we prove that the first $i$-dimensional total mass of $f$ is translation invariant.

\begin{lem}\label{red-v-0}
 Let $\xi_i\in \mathcal G_{i,n}$, $f=e^{-u},\,\,g=e^{-v}\in\mathcal A'$. Let $c=\inf u|_{\xi_i}=:u(0)$, $d=\inf v|_{\xi_i}:=v(0)$, and set $\widetilde{u}_i(x)=u|_{\xi_i}(x)-c$, $\widetilde v_i(x)=v|_{\xi_i}(x)-d$, $\widetilde\varphi_i(y)= (\widetilde u_i)^*(y)$, $\widetilde \psi_i(y)=(\widetilde v_i)^*(y)$, $\widetilde f_i=e^{-\widetilde u_i}$, $\widetilde g_i=e^{-\widetilde v_i}$,  and $\widetilde f_t|_i=\widetilde f\oplus t\cdot \widetilde g$.
Then if
$\lim\limits_{t\rightarrow 0^+}\frac{J_i(\widetilde f_t)-J_i(\widetilde f)}{t}=\int_{\xi_i}\widetilde \psi_i d\mu_i(\widetilde f)$
holds, then we have $\lim\limits_{t\rightarrow 0^+}\frac{J_i(f_t)-J_i(f)}{t}=\int_{\xi_i} \psi_i d\mu_i(f).$
\end{lem}
\begin{proof}
  By the construction, we have $\widetilde u_i(0)=0,\,\, \widetilde v_i(0)=0,$ and $\widetilde v_i\geq 0,\,\,\widetilde \varphi_i\geq 0,\,\,\widetilde \psi_i\geq 0$. Further, $\widetilde \psi_i(y)=\psi_i(y)+d$, and $\widetilde f_i=e^cf_i$. So
\begin{align}
\lim\limits_{t\rightarrow 0^+}\frac{J_i(\widetilde f_t)-J_i(\widetilde f)}{t}=\int_{\xi_i}\widetilde \psi_i d\mu_i(\widetilde f)=e^c\int_{\xi_i} \psi_i d\mu_i(f)+d e^c\int_{\xi_i} d\mu_i( f).
\end{align}
On the other hand, since $f_i\oplus t\cdot g_i=e^{-(c+dt)}(\widetilde f_i\oplus t\cdot\widetilde g_i),$
we have, $J_i(f\oplus t\cdot g)=e^{-(c+dt)}J_i(\widetilde f_i\oplus t\cdot\widetilde g_i).$
Derivation both sides of the above formula, we obtain
\begin{align*}
\lim_{t\rightarrow 0^+}\frac{J_i(f\oplus t\cdot g)-J_i(f)}{t}&=-d e^{-c}\lim_{t\rightarrow 0^+}J_i(\widetilde f_i\oplus t\widetilde g_i)dx+e^{-c}\lim_{t\rightarrow 0^+}\Big[\frac{J_i(\widetilde f_t)-J_i(\widetilde f)}{t}\Big]\\
&=-d e^{-c}J_i(\widetilde f_i)+\int_{\xi_i} \psi_i d\mu_i(f)+d \int_{\xi_i} d\mu_i( f)\\
&=\int_{\xi_i}\psi_i d\mu_i(f).
\end{align*}
So we complete the proof.
\end{proof}

\begin{thm}
  Let $f,\,\,g\in \mathcal A'$, and satisfy $-\infty \leq \inf(\log g)\leq +\infty$ and $W_i(f)>0$. Then $W_j(f,g)$ is defferentiable at $f$ along $g$, and it holds
\begin{align}
W_j(f,g)\in [-k,+\infty],
\end{align}
where $k=\max\{d,0\}W_i(f)$.
\end{thm}
\begin{proof}
Let $\xi_i\in\mathcal G_{i,n}$, since $u|_{\xi_i}:=-\log (f|_{\xi_i})=-(\log f)|_{\xi_i}$ and $v|_{\xi_i}:=-\log (g|_{\xi_i})=- (\log f)|_{\xi_i}.$
By the definition of $f_t$ and the Proposition \ref{f-add-mul} we obtain, $f_t|_{\xi_i}=(f\oplus t\cdot g)|_{\xi_i}=f|_{\xi_i}\oplus t\cdot g|_{\xi_i}.$
Notice that $v|_{\xi_i}(0)=v(0)$, set
$d:=v(0)$, $\widetilde v|_{\xi_i}(x):=v|_{\xi_i}(x)-d$, $\widetilde g|_{\xi_i}(x):=e^{-\widetilde v|_{\xi_i}(x)}$, and $\widetilde f_t|_{\xi_i}:=f|_{\xi_i}\oplus t\cdot \widetilde g|_{\xi_i}.$
Up to a translation of coordinates, we may assume $\inf(v)=v(0).$  The Lemma \ref{mon-ut-ft} says that for every $x\in \xi_i$,
$$f|_{\xi_i}\leq \widetilde f_t|_{\xi_i}\leq \widetilde f_1|_{\xi_i},\,\,\,\,\,\,\,\,\,\,\,\,\,\,\forall x\in \mathbb R^n,\,\,\forall t\in[0,1].$$
Then there exists $\widetilde f|_{\xi_i}(x):=\lim\limits_{t\rightarrow 0^+}\widetilde f_t|_{\xi_i}(x)$. Moreover, it holds $\widetilde f|_{\xi_i}(x)\geq f|_{\xi_i}(x)$ and $\widetilde f_t|_{\xi_i}$ is pointwise decreasing as $t\rightarrow 0^+$.  Lemma \ref{con-for-u} and Proposition \ref{clo-uv} show that
$f|_{\xi_i}\oplus t\cdot  \widetilde g|_{\xi_i}\in \mathcal A'$,  $\forall t\in [0,1].$
Then $J_i(f)\leq J_i(\widetilde f_t)\leq J_i(\widetilde f_1)$, $-\infty \leq J_i(f), J_i(\widetilde f_1)<\infty$. Hence, by the monotone and  convergence, we have $\lim_{t\rightarrow 0^+}W_i( \widetilde f_t)=W_i(\widetilde f).$
In fact, by definition  we have $\widetilde f_t|_{\xi_i}(x)=e^{-\inf\{u|_{\xi_i}(x-y)+tv|_{\xi_i}(\frac{y}{t})\}}$, and
$$-\inf\{u|_{\xi_i}(x-y)+tv|_{\xi_i}(\frac{y}{t})\}\leq - \inf u|_{\xi_i}(x-y)-t\inf v|_{\xi_i}(\frac{y}{t}),$$
Note that $-\infty \leq \inf(v|_{\xi_i})\leq +\infty$,  then $- \inf u|_{\xi_i}(x-y)-t\inf v|_{\xi_i}(\frac{y}{t})$  is continuous function of variable $t$, then
\begin{align}\label{con-fun-ut}
\widetilde f|_{\xi_i}(x):=\lim_{t\rightarrow 0^+}\widetilde f_t|_{\xi_i}(x)=f|_{\xi_i}(x).
\end{align}
Moreover, $W_i(\widetilde f_t)$ is continuous function of $t$ $(t\in[0,1])$, then
 $\lim\limits_{t\rightarrow 0^+}W_i(\widetilde f_t)=W_i(f).$
Since $f_t|_{\xi_i}=e^{-dt}\widetilde f|_{\xi_i}(x)$, we have
\begin{align}\label{que-inv-two}
 \frac{ W_i(f_t)-W_i(f)}{t}=W_i(f)\frac{e^{-dt}-1}{t}+e^{-dt}\frac{W_i(\widetilde f_t)-W_i(f)}{t}.
\end{align}
Notice that, $\widetilde f_t|_{\xi_i}\geq f|_{\xi_i}$, we have the following two cases,
 that is: $\exists t_0>0: W_i(\widetilde f_{t_0})=W_i(f) $ or $W_i(\widetilde f_{t})=W_i(f)$, $\forall t>0.$

For the first case, since $W_i(\widetilde f_t)$ is a monotone increasing function of $t$, it must holds $W_i(\widetilde f_t)=W_i(f)$ for every $t\in [0,t_0]$. Hence we have $\lim\limits_{t\rightarrow 0^+}\frac{W_i(f_t)-W_i(f)}{t}=-dW_i(f)$, the statement of the theorem holds true.

In the latter case, since $\widetilde f_t|_{\xi_i}$ is increasing non-negative function, it means that $\log(W_i(\widetilde f_t))$ is an increasing concave function of $t$.
Then, $\exists \frac{\log(W_i(\widetilde f_t))-\log(W_i( f))}{t}\in [0,+\infty].$
On the other hand, since
\begin{align*}
\log'\Big(W_i(\widetilde f_t)\Big)\Big|_{t=0}=\lim_{t\rightarrow 0^+}\frac{\log(W_i(\widetilde f_t))-\log(W_i( f))}{W_i(\widetilde f_t)-W_i(f)}=\frac{1}{W_i (f)}.
\end{align*}
Then
\begin{align}\label{qer-lim-two}
\lim_{t\rightarrow 0^+}\frac{W_i(\widetilde f_t)-W_i(f)}{\log(W_i(\widetilde f_t))-\log(W_i( f))}=W_i(f)>0.
\end{align}
From above we infer that $\exists\lim_{t\rightarrow 0^+}\frac{W_i(\widetilde f_t)-W_i(f)}{t}\in[0,+\infty].$
Combining the above formulas we obtain
\begin{align*}
 \lim_{t\rightarrow 0^+}\frac{W_i(f_t)-W_i(f)}{t}\in[-\max\{d,0\}W_i(f),+\infty].
\end{align*}
So we complete the proof.
\end{proof}

In view of the example of the mixed Quermassintegral, it is natural to ask whether in general, $W_i(f,g)$ has some kind of integral representation.

\begin{defn}
  Let $\xi_i\in\mathcal G_{i,n}$ and $f=e^{-u}\in \mathcal A'$. Consider the gradient map $\nabla u: \mathbb R^n \rightarrow \mathbb R^n$, the Borel measure $\mu_i(f)$ on $\xi_i$ is defined by
\begin{align*}
  \mu_i(f):=\frac{(\nabla u|_{\xi_i})_\sharp}{\|x\|^{n-i}}(f|_{\xi_i}).
\end{align*}
\end{defn}

Recall that the following Blaschke-Petkantschin formula is useful.
\begin{prop}[\cite{jen-loc1998}]
Let $\xi_i\in \mathcal G_{i,n}$ $(i=1, \,2,\cdots ,n)$ be linear subspace of $\mathbb R^n$, $f$ be a non-negative bounded Borel function on $\mathbb R^n$, then
\begin{align}
\int_{\mathbb R^n}f(x)dx=\frac{\omega_n}{\omega_i}\int_{\mathcal G_{i,n}}\int_{\xi_i}f(x)\|x\|^{n-i}dxd\mu(\xi_i).
\end{align}
\end{prop}

Now we give a proof of Theorem \ref{thm-mix-quer}.
\begin{proof}[Proof of Theorem \ref{thm-mix-quer}]
By the definition of the $i$-th Quermassintegral of $f$, we have
$$\frac{W_{i}(f_t)-W_i(f)}{t}=\frac{\omega_n}{\omega_{n-i}}\int_{\mathcal G_{n-i,n}}\frac{J_{n-i}(f_t)-J_{n-i}(f)}{t}d\mu(\xi_{n-i}).$$
Let $t>0$ be fixed, take $C\subset \subset \Omega|_{\xi_{n-i}}$, and by reduction $0\in int (\Omega)|_{\xi_{n-i}}$, we have $C\subset \subset \Omega|_{\xi_{n-i}}$, by Lemma
\ref{lem-duvp}, we obtain
\begin{align*}
\lim_{h\rightarrow 0}&\frac{J_{n-i}(f_{t+h})(x)-J_{n-i}(f_t(x))}{h}=\int_{\xi_{n-i}}\psi\big(\nabla u_t|_{\xi_{n-i}}(x)\big)f_{t}|_{\xi_{n-i}}(x)dx,
\end{align*}
where $\psi=h_{g|_{\xi_{n-i}}}=v|_{\xi_{n-i}}^*$. Then we have,
\begin{align*}
  \lim_{h\rightarrow 0}\frac{W_{i}(f_{t+h})-W_i(f_t)}{h}&=\frac{\omega_n}{\omega_{n-i}}\int_{\mathcal G_{n-i,n}}\int_{\xi_{n-i}}\frac{\psi\big(\nabla u_t|_{\xi_{n-i}}(x)\big)f_{t}|_{\xi_{n-i}}(x)}{\|x\|^{n-i}}\|x\|^{n-i}dxd\mu(\xi_{n-i}),\\
&=\int_{\mathbb R^n}\frac{\psi\big(\nabla u_t|_{\xi_{n-i}}(x)\big)f_{t}|_{\xi_{n-i}}(x)}{\|x\|^{n-i}}dx\\
&=\int_{\mathbb R^n}\psi d\mu_{n-i}(f_t).
\end{align*}
So we have, $W_{i}(f_{t+h})-W_i(f_t)=\int_0^t\Big\{\int_{\mathbb R^{n}}\psi d\mu_{n-i}(f_s)\Big\}ds.$
The continuous of $\psi$ implies
$\lim\limits_{s\rightarrow 0^+}\int_{\mathbb R^{n}}\psi d\mu_{n-i}(f_s)ds=\int_{\mathbb R^{n}}\psi  d\mu_{n-i}(f)ds.$
Therefore,
\begin{align*}
  \lim_{t\rightarrow 0^+}\frac{W_i(f_t)-W_i(f)}{t}&=\frac{d}{dt}W_i(f_t)|_{t=0^+}=\lim_{s\rightarrow 0^+}\frac{d}{dt}W_i(f_t)|_{t=s}\\
&=\lim_{s\rightarrow 0^+}\frac{d}{dt}\int_0^t\Big\{\int_{\mathbb R^{n}}\psi  d\mu_{n-i}(f_s)\Big\}ds\\
&=\int_{\mathbb R^n}\psi d\mu_{n-i}(f).
\end{align*}
Since $\psi= h_{g|_\xi}$, so we have
$$W_i(f,g)=\frac{1}{n-i}\lim\limits_{t\rightarrow 0^+}\frac{W_i(f_t)-W_i(f)}{t}=\frac{1}{n-i}\int_{\mathbb R^n}h_{g|_{\xi_{n-i}}} d\mu_{n-i}(f).$$
So we complete the proof.
\end{proof}

\begin{rem}
From the integral representation (\ref{mix-que-exp}),  the $i$-th functional mixed Quermassintegral is linear in its second argument, with the sum in $\mathcal A'$, for $f,\,g,\,h\in\mathcal A'$, then we have $W_i(f,g\oplus h)= W_i(f,g)+ W_i(f,h).$
\end{rem}

{\bf Data Availability:} No data were used to support this study.

{\bf Conflicts of Interest:} The authors declare no conflict of interest.

{\bf Authors¡¯ Contributions:} All authors contributed equally to this work. All authors have
read and agreed to the published version of this manuscript.\\

{\bf Acknowledgments}\\

The authors would like to strongly thank the anonymous referee for the very valuable comments and helpful suggestions that directly lead to improve the original manuscript.

\end{document}